\documentclass[10pt,a4paper]{article}
\usepackage[latin2]{inputenc}
\usepackage[T1]{fontenc}
\usepackage{amsmath,amssymb,multicol,amsthm}
\usepackage{textcomp,array,amsfonts}
\usepackage{graphicx,color,mathrsfs,txfonts}
\usepackage[all]{xy}
\theoremstyle{plain}
\newtheorem{thm}{\bfseries Theorem}[section]
\newtheorem{lem}[thm]{\bfseries Lemma}
\newtheorem{ex}[thm]{\bfseries Example}
\newtheorem{rem}[thm]{\bfseries Remark}
\newtheorem{prop}[thm]{\bfseries Proposition}
\newtheorem{cor}[thm]{\bfseries Corollary}
\date{}
\title{Configuration complexes and a variant of Cathelineau's complex in weight 3}
\author{Raziuddin Siddiqui \thanks{email: rdsiddiqui@fuuast.edu.pk} \\ \small \textit{Mathematical Sciences Research Centre}\\ 
\small \textit{Federal Urdu University, Karachi.} }
\begin{document}
\maketitle

\begin{abstract}
In this paper we consider the Grassmannian complex of projective configurations in weight 2 and 3, and Cathelineau's infinitesimal polylogarithmic complexes. Our main result is a morphism of complexes between the Grassmannian complex and the associated infinitesimal polylogarithmic complex.
%will form morphism of complexes between the Grassmannian complex and to a variant of the Cathelineau's complex and also to the Cathelineau's tangential complex. Our approach has several ingredients.

In order to establish this connection we introduce an $F$-vector space $\beta^D_2(F)$, which is an intermediate structure between a $\varmathbb{Z}$-module $\mathcal{B}_2(F)$ (scissors congruence group for $F$) and Cathelineau's $F$-vector space $\beta_2(F)$ which is an infinitesimal version of it. The structure of $\beta^D_2(F)$ is also infinitesimal but it has the advantage of satisfying similar functional equations as the group $\mathcal{B}_2(F)$. We put this in a complex to form a variant of Cathelineau's infinitesimal complex for weight 2.
\end{abstract}

\section{Introduction}
In his seminal papers (\cite{Gonc},\cite{Gonc1},\cite{Gonc2},\cite{Gonc3}), Goncharov uses the Grassmannian complex (first introduced by Suslin (see \cite{Sus2})) associated to points in $\varmathbb{P}^2$ in order to prove Zagier's conjecture on polylogarithms and special $L$-values (see \cite{Zag1}) for weight $n=3$. This conjecture in particular asserts that the values of Dedekind zeta function $\zeta_F(s)$ for some number field $F$ at an integer point $s=n\geq2$ can be expressed as a determinant of $n$-logarithms evaluated at points in  $F$. It was known for $n=2$ by work of Suslin, Borel and Bloch and also proved in a slightly weaker form by Zagier himself. Goncharov forms  an ingenious proof for weight $n=3$.

In the process, he introduces motivic complexes $\Gamma(n)$. Cathelineau investigates variants of these complexes in the additive (both infinitesimal and tangential) setting (see \cite{Cath1},\cite{Cath2},\cite{Cath3}). %Cathelineau constructed infinitesimal and tangential versions of complexes $\Gamma(n)$. 

One of the most important ingredients of Goncharov's work is the triple-ratio (Goncharov called it generalized cross-ratio) which is first introduced by Goncharov (see \cite{Gonc1}). In his earlier paper Goncharov had a formula (which is not visibly antisymmetric) for the morphism $f^{(3)}_2:C_6(3)\rightarrow \mathcal{B}_3(F)$, (see $\S$4 in \cite{Gonc}), for any field $F$, where $C_6(3)$ is the free abelian group generated by the configurations of 6 points in 3 dimensional $F$-vector space modulo the action of $GL_3(F)$ . Goncharov introduced the triple ratio by anti-symmetrization  of formula for $f^{(3)}_2$. Having defined the triple-ratio he described an antisymmetric formula for the morphism $f_6(3):C_6(3)\rightarrow \mathcal{B}_3(F)$, but with the restriction that it applies to generic configuration only, where points are in generic position (see Formula 3.9 in \cite{Gonc1}) (unfortunately, in \cite{Gonc1} there was a missing factor in the formula; this missing factor of $\frac{15}{2}$ was pointed out by Gangl and Goncharov provided a proof of the corrected formula in the appendix of \cite{Gonc4}). By using algebraic $K$-theory he constructed a map of complexes from the Grassmannian complex to his own complex and then he proved Zagier's conjecture for weight $n=3$.

Our point of view is to bring the geometry of configuration spaces into infinitesimal setting. We tried to find suitable  morphisms between the Grassmannian subcomplex $(C_*(n),d)$ (see diagram \eqref{Gonbicomp} in \ref{Gra_comp}) and Cathelineau's analogues of Goncharov's complexes $\Gamma(n)$. For weight $n=2$ and $n=3$, we proved that the corresponding diagrams in the infinitesimal setting connecting the Grassmannian subcomplex $(C_*(n),d)$ (see diagram \eqref{Gonbicomp} in $\S$\ref{Gra_comp}) are commutative (see $\S$\ref{add_tri}). 

Goncharov outlined the proof for commutativity of the left square of diagram made by Grassmannian complex and weight $n=3$ motivic complex (see $\S$3 in \cite{Gonc1} for the actual diagram and  appendix of \cite{Gonc4} for the proof).  For this he worked in $\bigwedge^2F^\times\otimes F^\times$, using the factorisation of $1-\frac{\Delta(l_0,l_1,l_3)\Delta(l_1,l_2,l_4)\Delta(l_2,l_0,l_5)}{\Delta(l_0,l_1,l_4)\Delta(l_1,l_2,l_5)\Delta(l_2,l_0,l_3)}$, where $\Delta(l_i,l_j,l_k)$ denotes some $3\times3$-determinant, into a $3\times3$-determinant and a $6\times6$-determinant and also had to appeal to a deeper result in algebraic $K$-theory (see Lemma 5.1 and Proposition 5.2 in \cite{Gonc4}).

We observe that each term in the triple-ratio can be rewritten as product of two ``projected'' cross-ratios in $\varmathbb{P}^2$, which enables us to give an elementary proof (which does not use algebraic $K$-theory) of our main result (Theorem \ref{claim3b}). 

Furthermore, we define infinitesimal group $\beta_2^D(F)$ for any derivative $D\in Der_{\varmathbb{Z}}F$ over a field $F$ which has more or less similar functional equations as the group $\mathcal{B}_2(F)$ and use it to our advantage for the proof which works almost same for the two direct summand involving $\beta^D_2(F)\otimes F^\times$ and $F\otimes \mathcal{B}_2(F)$. In summary, the proof of Theorem \ref{claim3b} consists of rewriting the triple-ratio as the product of two cross-ratios, combinatorial techniques and the use of functional equations in $\beta_2^D(F)$ and $\mathcal{B}_2(F)$. 

\section{Preliminaries and Background}
As we mentioned in the introduction, we are relating the Grassmannian complex to a variant of Cathelineau's complex. We will also present the variant of Cathelineau's (infinitesimal) complex in $\S$\ref{def1} and will try to form a generalized complex for $\beta_n^D(F)$ as Goncharov's work in \cite{Gonc3}.
%We can see a nice presentation of the relationship between Grassmannian complex and Bloch-Suslin complex which forms a bicomplex or Goncharov's complex if we go through the Goncharov's early 90's results (see \cite{Gonc},\cite{Gonc1},\cite{Gonc2}). He shows that the associated diagrams of complexes are commutative and gives an excellent interpretation of the popular Abel's five-term relation in terms of geometric configurations in projective space. We will try to form some similar diagrams here. We are trying to relate those Grassmannian complexes to Cathelineau's Complexes (\cite{Cath1},\cite{Cath2}). We will see that these diagrams of complexes are also commutative. We also want to see the shape of Cathelineau's four-term relation in terms of projective geometric configurations.
%Goncharov has already found the maps between Grassmannian complex and the Bloch-Suslin complex in the projective space (see \cite{Gonc,Gonc1,Gonc2}). In these papers he nicely shows that the associated bicomplexes are commutative and writes Abel's five term relation in terms of projective geometric configurations. This chapter is related to Cathelineau's complexes (see \cite{Cath1,Cath2}) and we will try to connect the Grassmannian complex to the Cathelineau's complex also will try to write popular four-term relation by using geometric configurations. 
 
\subsection{Grassmannian complex}\label{Gra_comp}
In this section, we recall concepts from (see \cite{Gonc}, \cite{Gonc2}).
Consider $\tilde{C}_m(X)$, which is the free abelian group generated by elements $(x_1,\ldots,x_m)\in X^m$ for some set $X$ with $x_i \in X$. Then we have a simplicial complex $(\tilde{C}_\ast(X),d)$ generated by simplices whose vertices are the elements of $X$, where the differential in degree -1 is given on generators by 
\[d: C_m(X)\rightarrow C_{m-1}(X)\]

\begin{equation}\label{ddef}
d\colon(x_1,\ldots,x_{m})\mapsto\sum_{i=0}^m (-1)^i(x_1,\ldots,\hat{x_i},\ldots,x_{m})\notag
\end{equation}
Let $G$ be a group acting on $X$. The elements of $G\setminus X^{m}$ are called configurations of $X$, where $G$ is acting diagonally on $X^m$. Further assume that $C_m(X)$ is the free abelian group generated by the configurations of $m$ elements of $X$ then there is a complex $(C_\ast(X),d)$, and $\tilde{C}_\ast(X)_G$ be the group of coinvariants of the natural action of G on $C_\ast(X)=\tilde{C}_\ast(X)$. For $m> n$, let us define $C_m(n)$ (or $C_m(\varmathbb{P}^{n-1}_F)$) which is the free abelian group, generated by the configurations of $m$ vectors in an $n$-dimensional vector space $V_n=\varmathbb{A}^n_F$  over a field $F$ (any $n$ vectors arising by using $X=V_n$) (or $m$ points in $\varmathbb{P}^{n-1}_F$) in generic position (an $m$-tuple of vectors in an $n$-dimensional vector space $V_n$ is in generic position if $n$ or fewer number of vectors are linearly independent). Apart from the above differential $d$, we have another differential map:
\[d':C_{m+1}(n+1)\rightarrow C_{m}(n)\]
\[d'\colon(l_0,\ldots,l_{m})\mapsto\sum_{i=0}^m (-1)^i(l_i|l_0,\ldots,\hat{l_i},\ldots,l_{m}),\]
where $(l_i|l_0,\ldots,\hat{l_i},\ldots,l_{m})$ is the configuration of vectors in $V_{n+1}/\langle l_i\rangle$ defined as the $n$-dimensional quotient space, obtained by the projection of vectors $l_j\in V_{n+1}$, $j\neq i$, projected from $C_{m+1}(n+1)$ to $C_m(n)$ from which we have the following bicomplex
\begin{displaymath}\label{Gonbicomp}
\xymatrix{
 &&\vdots\ar[d]^{}	&\vdots\ar[d]	&\vdots\ar[d] \\
&\cdots\ar[r]	&C_{n+5}(n+2)\ar[d]^{d'}\ar[r]^{d}  &C_{n+4}(n+2)\ar[d]^{d'}\ar[r]^{d}    &C_{n+3}(n+2)\ar[d]^{d'}\\
&\cdots\ar[r]	&C_{n+4}(n+1)\ar[d]^{d'}\ar[r]^{d} &C_{n+3}(n+1)\ar[d]^{d'}\ar[r]^{d}	  &C_{n+2}(n+1)\ar[d]^{d'}\\
&\cdots\ar[r]   &C_{n+3}(n)\ar[r]^{d}	&C_{n+2}(n)\ar[r]^{d}	&C_{n+1}(n)
}\tag{2.1a}
\end{displaymath}
which is called the Grassmannian bicomplex. 
% We are reducing this calculation
%We will verify here commutativity of the above diagram for this we just need to show that $d'\circ d=d\circ d'$ for the group $C_{n+k+m}(n+k)$.
%\[d'\circ d(l_0,\ldots,l_{n+k+m-1})=\sum_{i=0}^{n+k+m-1}(-1)^i\left\lbrace \sum_{\substack{j=0\\ j\neq i}}^{n+k+m-1}(-1)^j(l_j|l_0,\ldots,\hat{l}_i,\ldots,\hat{l}_j,\ldots,l_{n+k+k+m-1})\right\rbrace\]
%and
%\[d\circ d'(l_0,\ldots,l_{n+k+m-1})=\sum_{i=0}^{n+k+m-1}(-1)^i\left\lbrace \sum_{\substack{j=0\\ j\neq i}}^{n+k+m-1}(-1)^j(l_i|l_0,\ldots,\hat{l}_i,\ldots,\hat{l}_j,\ldots,l_{n+k+m-1})\right\rbrace\]
For the following we will use a subcomplex $(C_\ast(n),d)$ called the Grassmannian complex, of the above
\[\cdots\xrightarrow{d} C_{n+3}(n)\xrightarrow{d} C_{n+2}(n)\xrightarrow{d}C_{n+1}(n)\]
We concentrate our studies to the subcomplex $(C_\ast(n),d)$, but in some cases we will also use the following subcomplex $(C_*(*),d')$ of the Grassmannian complex
\[\cdots\xrightarrow{d'} C_{n+3}(n+2)\xrightarrow{d'} C_{n+2}(n+1)\xrightarrow{d'}C_{n+1}(n)\]

\subsection{Polylogarithmic Groups}
From now on we will denote our field by $F$ and $F-\{0,1\}$ will be abbreviated as $F^{\bullet\bullet}$. In some texts $F^{\bullet\bullet}$ is also referred as doubly punctured affine line over $F$ in (\cite{PandG}). We will also denote $\varmathbb{Z}[\varmathbb{P}^1_F]$ as the free abelian group generated by $[x]$ where $x\in \varmathbb{P}^1_F$.

\textbf{Scissors congruence group:}(\cite{Sus1})The Scissors congruence group $\mathcal{B}(F)$ of $F$ is defined as the quotient of the free abelian group $\varmathbb{Z}[F^{\bullet\bullet}]$ by the subgroup generated by the elements of the form 
\[ [x]-[y]+\left[\frac{y}{x}\right]-\left[\frac{1-y}{1-x}\right]+\left[\frac{1-y^{-1}}{1-x^{-1}}\right] \text{ where } x\neq y,\hspace{3pt}x,y\neq 0,1\]
The above relation is the famous Abel's five-term relation for the dilogarithm. 
%It can also be interpreted geometrically (in terms of scissors congruences) whence its name: Consider a an ideal polyhedron hyperbolic 3-space with five vertices $x_1,\ldots,x_5$. Divide this polyhedron into five tetrahedra by leaving out one vertex at a time i.e $\{x_2,x_3,x_4,x_5\}$ and $\{x_1,x_3,x_4,x_5\}$ with common face $\{x_3,x_4,x_5\}$ and three other tetrahedra $\{x_1,x_2,x_4,x_5\}$, $\{x_1,x_2,x_3,x_5\}$ and $\{x_1,x_2,x_3,x_4\}$ so that the sum of first two volumes is same as the sum of last three volumes( when taken with the right orientation). This volumes identity is mimicked in the following relation (where $r(a,b,c,d)$ denotes the cross-ratio of four points)
%\[[r(x_2,x_3,x_4,x_5)]+[r(x_1,x_3,x_4,x_5)]=[r(x_1,x_2,x_4,x_5)]+[r(x_1,x_2,x_3,x_5)]+[r(x_1,x_2,x_3,x_4)]\]
%This relation is a version of the above five-term relation.
\subsubsection{Bloch-Suslin and Goncharov's polylog complexes}\label{b_s}
In this section we will closely follow \cite{Gonc} and \cite{Gonc1}.
\begin{enumerate}

\item \textbf{Weight 1:}
We define subgroup $R_1(F)\subset \varmathbb{Z}[\varmathbb{P}^1_F]$ by 
\[R_1(F)= \left\langle [xy]-[x]-[y], x,y\in F^\times-\{1\}\right\rangle\]
The map $\delta_1:\mathcal{B}_1(F)\rightarrow F^\times,[a]\mapsto a$ is defined as an isomorphism (see $\S$1 of \cite{Gonc}), so we have $\mathcal{B}_1(F)= F^\times$.
\item \textbf{Weight 2:}
First we define the subgroup $R_2(F)\subset \varmathbb{Z}[\varmathbb{P}^1_F\setminus \{0,1,\infty\}]$
\[R_2(F):=\left\langle \sum^4_{i=0}(-1)^i[r(x_0,\ldots,\hat{x}_i,\ldots,x_4)],\quad x_i \in \varmathbb{P}^1_F \right\rangle\]
where $r(x_0,x_1,x_2,x_3)=\frac{(x_0-x_3)(x_1-x_2)}{(x_0-x_2)(x_1-x_3)}$ is the cross-ratio of four points and $\delta_2$ is defined 
as
\[\delta_2:\varmathbb{Z}[\varmathbb{P}^1_F\setminus\{0,1,\infty\}]\rightarrow \bigwedge{}^2F^\times \]
\begin{equation*}
[x]\mapsto(1-x)\wedge x 
\end{equation*}
where $\bigwedge{}^2F^\times=F^\times\otimes_{\varmathbb{Z}}F^\times/\langle x \otimes_{\varmathbb{Z}} x| x\in F^\times\rangle$. One has $\delta_2\left(R_2(F)\right)=0$. Now we can define the free abelian group $\mathcal{B}_2(F)$ which is generated by $[x] \in \varmathbb{Z}[\varmathbb{P}^1_F\setminus \{0,1,\infty\}]$ and quotient by the subgroup $R_2(F)\subset \varmathbb{Z}[\varmathbb{P}^1_F\setminus \{0,1,\infty\}]$, i.e. \[\mathcal{B}_2(F)=\frac{\varmathbb{Z}[\varmathbb{P}^1_F\setminus \{0,1,\infty\}]}{R_2(F)}\] and we get a complex $B_F(2)$ called the Bloch-Suslin complex of $F$
\[B_F(2):\quad\quad\quad \mathcal{B}_2(F)\xrightarrow{\delta} \bigwedge{}^2F^\times\]
where first term is in degree 1 and second term in degree 2 and $\delta$ is induced from $\delta_2$ due to fact $\delta_2\left(R_2(F)\right)=0$.
\item \textbf{Weight 3:}
Consider the triple-ratio of six points $r_3\in \mathbb{Z}[\varmathbb{P}^1_F]$ which is defined as $r_3: C_6(\varmathbb{P}^2_F)\rightarrow \varmathbb{Z}[\varmathbb{P}^1_F]$, where $C_6(\varmathbb{P}^2_F)$ is a free abelian group generated by the configurations of 6 points in generic position over $\varmathbb{P}^1_F$
\[r_3(l_0,\ldots,l_5)=\text{Alt}_6\left[\frac{\Delta(l_0,l_1,l_3)\Delta(l_1,l_2,l_4)\Delta(l_2,l_0,l_5)}{\Delta(l_0,l_1,l_4)\Delta(l_1,l_2,l_5)\Delta(l_2,l_0,l_3)}\right]\]
where $l_i$ is the point in $\varmathbb{P}^2_F$, $\Delta(l_i,l_j,l_k)=\langle \omega,l_i\wedge l_j \wedge l_k\rangle$ and $\omega \in \det V^*$. Now define the relation $R_3(F) \in \varmathbb{Z}[\varmathbb{P}^1_F]$ 
\[R_3(F):=\left\langle \sum^6_{i=0}(-)^ir_3(l_0,\ldots,\hat{l}_i,\ldots,l_6)\Big|\quad (l_0,\ldots,\hat{l}_i,\ldots,l_6)\in C_6(\varmathbb{P}^2_F)\right\rangle\]
One can define $\mathcal{B}_3(F)$ as the free abelian group generated by $[x] \in \mathbb{Z}[\varmathbb{P}^1_F]$ and quotient by $R_3(F)$, $[0]$ and $[\infty]$. Thus we get the complex $B_F(3)$
\[B_F(3):\quad\quad\quad \mathcal{B}_3(F)\xrightarrow{\delta}\mathcal{B}_2(F)\otimes_{\varmathbb{Z}} F^\times \xrightarrow{\delta} \bigwedge{}^3 F^\times\]
\item \textbf{Weight $\geq3$:}
Here we will define group $\mathcal{B}_n(F)$. Suppose $\mathcal{R}_n(F)$ is defined already, we set
\[\mathcal{B}_n(F)=\frac{\varmathbb{Z}[\varmathbb{P}^1_F]}{\mathcal{R}_n(F)}\]
and the morphism
\[\delta_n:\varmathbb{Z}[\varmathbb{P}^1_F]\rightarrow\mathcal{B}_{n-1}(F)\otimes F^\times\]
\begin{equation*}
[a]\mapsto
\begin{cases}
0 &\textrm{ if }x=0,1,\infty \notag\\
[x]_{n-1}\wedge x &\textrm{ otherwise}
\end{cases}
\end{equation*}
where $[x]_n$ is class of $[x]$ in $\mathcal{B}_n(F)$. We find more important is the case for $n\geq2$, where we define 
\[\mathcal{A}_n(F)=\ker\delta_n\]
and $\mathcal{R}_n(F)\subset \mathbb{Z}[\varmathbb{P}^1_F]$ is generated by the elements $\alpha(0)-\alpha(1)$,$[\infty]$ and $[0]$, where $\alpha(t)$ runs through all the elements of $\mathcal{A}_n(F(t))$, for an indeterminate $t$.
\end{enumerate}
\begin{lem}\label{gonc1.16}(Goncharov)
For $n\geq 2$, $\mathcal{R}_n(F)\subset\ker \delta_n$
\end{lem}
Proof: See lemma 1.16 of \cite{Gonc}.\hfill $\Box$

Goncharov defines the following complex (\cite{Gonc},\cite{Gonc1}) for the group $\mathcal{B}_n(F)$.
\begin{align}\label{Gonc_comp}
\mathcal{B}_n(F)\xrightarrow{\delta}\mathcal{B}_{n-1}\otimes F^\times\xrightarrow{\delta}\mathcal{B}_{n-2}\otimes \bigwedge{}^2F^\times\xrightarrow{\delta}\cdots\xrightarrow{\delta}\mathcal{B}_2(F)\bigwedge{}^{n-2}F^\times\xrightarrow{\delta}\frac{\bigwedge{}^nF^\times}{2-\text{torsion}}
\end{align}

%Here one can define group $B_n(F)$ as the first cohomology group of the above complex.
%\[B_n(F)=\frac{\mathcal{A}_n(F)}{\mathcal{R}_n(F)}\]
%so $B_n(F)\subset \mathcal{B}_n(F)$

\subsection{Infinitesimal Complexes (Cathelineau's Complexes)}\label{cat_comp}
\indent There are two versions of infinitesimal complex or infinitesimal groups. In the literature the first one was introduced by Cathelineau \cite{Cath2} while the other version was introduced by Bloch-Esnault \cite{BandE1} also called ``additive``. The latter version is beyond the  scope of this text we will discuss here only the former one.

Cathelineau (\cite{Cath2},\cite{Cath1}) has defined the group ( in fact an $F$-vector space) as an infinitesimal analogue of Goncharov's groups $\mathcal{B}_n(F)$ as follows 
\begin{enumerate}

\item We define $\beta_1(F)=F$

\item One can define $\beta_2(F)$ as
\[\beta_2(F)=\frac{F[F^{\bullet\bullet}]}{r_2(F)}\]
where $r_2(F)$ is the kernel of the map
\[\partial_2: F[F^{\bullet\bullet}]\rightarrow F\otimes_{\varmathbb{Z}} F^\times\]
\[[a]\mapsto a\otimes_{F} a+(1-a)\otimes_{\varmathbb{Z}}(1-a)\]
Cathelineau \cite{Cath2} has shown that $r_2(F)$ is given as the subvector space of $F[F^{\bullet\bullet}]$ spanned by the elements
\[[a]-[b]+a\left[\frac{b}{a}\right]+(1-a)\left[\frac{1-b}{1-a}\right], a,b \in F^{\bullet\bullet}, a\neq b,\]
hence passing to the quotient by $r_2(F)$ we obtain the complex
\begin{align}\label{catcomp2}
\beta_2(F)\xrightarrow{\partial}&F\otimes_{\varmathbb{Z}} F^\times
\end{align}
\[\partial:\langle a\rangle_2\mapsto a\otimes a+(1-a)\otimes (1-a)\]
\item For $n\geq 3$, the $F$-vector space $\beta_n(F)$ is defined as
\[\beta_n(F)=\frac{F[F^{\bullet\bullet}]}{r_n(F)}\]
where $r_n(F)$ is kernel of the map
\[\partial_n:F[F^{\bullet\bullet}]\rightarrow \left(\beta_{n-1}(F)\otimes F^\times\right)\oplus\left(F\otimes\mathcal{B}_{n-1}(F) \right)\]
\[ [a]\mapsto \langle a\rangle_{n-1}\otimes a+(1-a)\otimes[a]_{n-1} \]
where $\langle a\rangle_k$ is the class of $[a]$ in $\beta_k(F)$ and $[a]_k$ is the class of $[a]$ in $\mathcal{B}_k(F)$.
For $n=2$, we have the following complex of $F$-vector spaces.
\begin{align}\label{catcomp3}
\beta_3(F)\xrightarrow{\partial}(\beta_2(F)\otimes F^\times)\oplus&(F\otimes \mathcal{B}_2(F))\xrightarrow{\partial}F\otimes\bigwedge{}^2F^\times
\end{align}
where
\[\partial:\langle a\rangle_3\mapsto \langle a\rangle_2\otimes a+(1-a)\otimes[a]_2\]
\[\partial:\langle a\rangle_2\otimes b+x\otimes [y]_2\mapsto -\left(a\otimes a\wedge b+(1-a)\otimes(1-a)\wedge b\right)+x\otimes (1-y)\wedge y\]
\end{enumerate}
Before the following lemma we shall introduce K\"{a}hler differentials (see $\S$25 in \cite{Mats} and $\S$26 in \cite{Mats1}). First, recall the definition of a derivation map $D\in Der(A,M)$ for a ring $A$ and an $A$-module $M$ is $D:A\rightarrow M$ and this map satisfies $D(a+b)=D(a)+D(b)$ and $D(ab)=aD(b)+bD(a)$. Now an $A$-module $\Omega_{A/F}$ is generated by $\{da{}|{}a\in A\}$ so that the uniqueness of a linear map $f:\Omega_{A/F}\rightarrow M$ satisfying  $D=f\circ d$ is obvious (see p192 of \cite{Mats}). If $a\in A$ then the element $da\in\Omega_{A/F}$ and called the differential of $a$ and the $A$-module $\Omega_{A/F}$ is called the module of K\"{a}hler differentials.
\begin{lem}\label{catexa}(Cathelineau \cite{Cath1},\cite{Cath2})
The complexes \ref{catcomp2} and \ref{catcomp3} are quasi-isomorphic to $\Omega^i_F$ through the maps $d\log:\bigwedge^iF^\times\rightarrow \Omega^i_F$ so that the following sequences 
\[0\rightarrow\beta_2(F)\xrightarrow{\partial}F\otimes F^\times \xrightarrow{d\log}\Omega^1_F\rightarrow 0\]
\[0\rightarrow\beta_3(F)\xrightarrow{\partial}(\beta_2(F)\otimes F^\times)\oplus(F\otimes B_2(F))\xrightarrow{\partial}F\otimes\wedge^2F^\times\xrightarrow{d\log}\Omega_F^2\rightarrow 0\]
are exact. Here $\Omega^i_F$ is the vector space of K\"{a}hler differential with the respective definitions of $d\log$ as $d\log(a\otimes b)=a\frac{db}{b}$ and $d\log(a\otimes b\wedge c)=a\frac{db}{b}\wedge\frac{dc}{c}$.
\end{lem}
\subsubsection{Functional equations in $\beta_2(F)$}
Here we will mainly focus on the work in (\cite{PandG})
\begin{enumerate}
\item The two-term relation
\[\langle a\rangle_2=\langle 1-a\rangle_2\]
\item The inversion relation.
\[\langle a\rangle_2=-a\left\langle\frac{1}{a}\right\rangle_2\]
\item The distribution relation
\[\langle a^m\rangle_2=\sum_{\zeta^m=1}\frac{1-a^m}{1-\zeta a}\langle \zeta a\rangle_2\]
\item The four-term relation in $F[F^{\bullet\bullet}]$.
\begin{equation}\label{catfour}
\langle a\rangle_2-\langle b\rangle_2+a\left\langle \frac{b}{a}\right\rangle_2+(1-a)\left\langle \frac{1-b}{1-a}\right\rangle_2=0,\quad a\neq b\notag
\end{equation}
The above equation is an infinitesimal version of the famous five-term relation and it can be deduced directly from the following form of five term relation \cite{Sus1}.
\end{enumerate}
\[[a]_2-[b]_2+\left[\frac{b}{a}\right]_2-\left[\frac{1-b}{1-a}\right]_2+\left[\frac{1-\frac{1}{b}}{1-\frac{1}{a}}\right]_2=0\]
%If $*:F[F^{\bullet\bullet}]\rightarrow F[F^{\bullet\bullet}]$ is defined as $*(a)=a(1-a)$ then we compute the coefficient of $*(x)$ where $x\in F[F^{\bullet\bullet}]$.
%For first term
%\[(a)'=a(1-a)=1\cdot *(a)=1\cdot\langle a\rangle_2\]
%for second term 
%\[(b)'=b(1-b)=1\cdot *(b)=1\cdot\langle a\rangle_2\]
%for third term
%\begin{align*}
%\left(\frac{b}{a}\right)'=&\frac{ab'-ba'}{a^2}=\frac{ab(1-b)-ba(1-a)}{a^2}\\
%=&a\cdot \frac{b}{a}\left(1-\frac{b}{a}\right)=a\cdot *\left(\frac{b}{a}\right)=a\cdot\left\langle \frac{b}{a}\right\rangle_2
%\end{align*}
%for fourth term
%\begin{align*}
%\left(\frac{1-b}{1-a}\right)'=&\frac{(1-a)(1-b)'-(1-b)(1-a)'}{(1-a)^2}\\
%=&\frac{(1-a)b(1-b)-(1-b)a(1-a)}{(1-a)^2}\\
%=&-(1-a)\cdot\left(\frac{1-b}{1-a}\right)\left(1-\frac{1-b}{1-a}\right)\\
%=&-(1-a)\cdot *\left(\frac{1-b}{1-a}\right)=-(1-a)\left\langle\frac{1-b}{1-a}\right\rangle_2
%\end{align*}
%for fifth term 
%\begin{align*}
%\left(\frac{1-\frac{1}{b}}{1-\frac{1}{a}}\right)'=&\frac{\left(1-\frac{1}{a}\right)\cdot\left(1-\frac{1}{b}\right)'-\left(1-\frac{1}{b}\right)\cdot \left(1-\frac{1}{a}\right)'}{\left(1-\frac{1}{a}\right)^2}\\
%=&\frac{\left(1-\frac{1}{a}\right)\cdot\left(1-\frac{1}{b}\right)-\left(1-\frac{1}{b}\right)\cdot \left(1-\frac{1}{a}\right)}{\left(1-\frac{1}{a}\right)^2}\\
%=&0
%\end{align*}
%This is usual way to get four-term relation but later (see example \ref{four_term_ex} in next chapter) we will use the derivation technique to get same relation.

\subsubsection{Functional equation in $\beta_3(F)$}
Here as well we will mainly focus on the work of (\cite{PandG})
\begin{enumerate}
\item The three-term relation.
\begin{equation}\label{threeterm}
\langle 1-a \rangle_3-\langle a\rangle_3+a\left\langle 1-\frac{1}{a}\right\rangle_3=0\notag
\end{equation}
\item The inversion relation.
\begin{align}
\langle a\rangle_3=-a\left\langle \frac{1}{a}\right\rangle_3\notag
\end{align}
The inversion relation is a consequence of the three-term relation (\ref{threeterm}) (see lemma 3.11 of \cite{PandG}).

\item The distribution relation
\[\langle a^m\rangle_3=m\sum_{\zeta^m=1}\frac{1-a^m}{1-\zeta a}\langle \zeta a\rangle_3\]
\item The 22-term relation.(\cite{PandG})

There are number of ways to write it and one of them is the following.
\begin{align}\label{22_term}
&c\langle a\rangle_3 -c\langle b \rangle_3+(a-b+1)\langle c \rangle_3\notag\\
+&(1-c)\langle 1-a\rangle_3-(1-c)\langle1-b\rangle_3+(b-a)\langle1-c\rangle_3\notag\\
-&a\left\langle \frac{c}{a}\right\rangle_3+b\left\langle \frac{c}{b}\right\rangle_3+ca\left\langle \frac{b}{a}\right\rangle_3\notag\\
-&(1-a)\left\langle \frac{1-c}{1-a}\right\rangle_3+(1-b)\left\langle \frac{1-c}{1-b}\right\rangle_3+c(1-a)\left\langle \frac{1-b}{1-a}\right\rangle_3\notag\\
+&c(1-a)\left\langle \frac{a(1-c)}{c(1-a)}\right\rangle_3-c(1-b)\left\langle \frac{b(1-c)}{c(1-b)}\right\rangle_3-b\left\langle\frac{ca}{b}\right\rangle_3\notag\\
+&(1-c)a\left\langle \frac{a-b}{a}\right\rangle_3+(1-c)(1-a)\left\langle \frac{b-a}{1-a}\right\rangle_3\notag\\
-&(a-b)\left\langle \frac{(1-c)a}{a-b}\right\rangle_3-(1-b)\left\langle \frac{c(1-a)}{1-b}\right\rangle_3\notag\\
-&(b-a)\left\langle \frac{(1-c)(1-a)}{b-a}\right\rangle_3+c(a-b)\left\langle \frac{(1-c)b}{c(a-b)}\right\rangle_3\notag\\
+&c(b-a)\left\langle\frac{(1-c)(1-b)}{c(b-a)}\right\rangle_3=0\notag
\end{align}
\end{enumerate}

\subsection{Derivation in $F$-vector space}\label{def1}
Let $F$ be a field and $D \in Der_{\varmathbb{Z}}(F,F)$ be an absolute derivation, (see $\S$25 of \cite{Mats} and $\S$6 of \cite{PandG}) we will also write simply as $D \in Der_{\varmathbb{Z}}(F)$. For example if $x\in F$ then its derivative over $\varmathbb{Z}$ will be represented by $D(x)$ and will be an element of $F$ as well.

According to $\S$6.1 in \cite{PandG} we have $\tilde{f}_D:\mathbb{Z}[F]\rightarrow F[F^{\bullet\bullet}], [a]\mapsto\frac{D(a)}{a(1-a)}[a]$ induces a map
\[\tau_{2,D}:\mathcal{B}_2(F)\rightarrow\beta_2(F),[a]_2\mapsto \frac{D(a)}{a(1-a)}\langle a\rangle_2\]
We define an $F$-vector space $\beta^D_2(F)$ generated by $\llbracket a\rrbracket^D$ for $a\in F^{\bullet\bullet}$ and subject to the five-term relation 
\[\llbracket a\rrbracket^D-\llbracket b\rrbracket^D+\left\llbracket \frac{b}{a}\right\rrbracket^D-\left\llbracket \frac{1-b}{1-a}\right\rrbracket^D+\left\llbracket \frac{1-b^{-1}}{1-a^{-1}}\right\rrbracket^D\text{ where } a\neq b,\quad 1-a\neq 0,\]
where  $\llbracket a\rrbracket^D:=\frac{D(a)}{a(1-a)}[a]$ and $[a]\in F[F^{\bullet\bullet}]$. Furthermore, we have
\[\partial^D_2:F[F^{\bullet\bullet}]\rightarrow F\otimes F^\times\]
with
\[\partial^D_2:\llbracket a\rrbracket^D\mapsto -D\log(1-a)\otimes a+D\log(a)\otimes (1-a),\]
where $D\log a=\frac{D(a)}{a}$. We identify Im$(\tau_{2,D})\left(\subset\beta_2(F)\right)$ with $\beta_2^D(F)$. We can also write a variant of Cathelineau's complex by using the $F$-vector space \[\beta^D_2(F)\subset F[F^{\bullet\bullet}]/(\text{five-term relation}),\]
as
\[\beta^D_2(F)\xrightarrow{\partial^D}F\otimes F^\times \]
with
\[\partial^D:\llbracket a\rrbracket^D_2\mapsto -D\log(1-a)\otimes a+D\log(a)\otimes (1-a)\]
where $\llbracket a\rrbracket^D_2=\frac{D(a)}{a(1-a)}\langle a\rangle_2$.

We also want to define $F$-vector spaces $\beta_n^D(F)$ for $n\geq3$. For this we use a slightly different construction by Cathelineau which in the case $n=2$ gives his $\textbf{b}_2(F)$ (see \cite{Cath2}). For this he divides $F[F^{\bullet\bullet}]$ by the kernel of the map $\partial_2$, of which an important element is the Cathelineau's four-term relation. By Remark \ref{fifth} below the differential of the five-term relation in $\mathcal{B}_2(F)$ leads to Cathelineau's four-term relation. For later purpose we note that the differential of Goncharov's 22-term relation in $\mathcal{B}_3(F)$ vanishes in $\beta_3(F)$ for any $D\in Der_{\varmathbb{Z}}(F)$ (see Proposition 6.10 of \cite{PandG}). We define
\[\beta^D_3(F)=\frac{F[F^{\bullet\bullet}]}{\rho^D_3(F)}\]
where $\rho^D_3(F)$ is the kernel of the map
\[\partial^D_3:\llbracket a \rrbracket^D\mapsto \llbracket a \rrbracket^D_2\otimes a + D\log(a)\otimes [a]_2\]
Now we have an $F$-vector space $\beta_2^D(F)$ which is an intermediate stage between a $\varmathbb{Z}$-module $\mathcal{B}_2(F)$ and an $F$-vector space $\beta_2(F)$ and has two-term and inversion relations same as $\mathcal{B}_2(F)$.
\subsection{Functional Equations in $\beta^D_2(F)$}\label{func_beta}
The inversion and two-term relations in $\beta^D_2(F)$ are quite similar to group $\mathcal{B}_2(F)$.

1. Two-term relation: 
\[\llbracket a\rrbracket^D_2=-\llbracket 1-a\rrbracket^D_2\]
% We are reducing calculation here
%We know from Cathelineau's $F$-vector space $\beta_2(F)$.
%\begin{align}
%\langle a\rangle_2&=\langle 1-a\rangle_2\notag\\
%\frac{D(a)}{a(1-a)}\langle a\rangle_2&=\frac{D(a)}{a(1-a)}\langle 1-a\rangle_2\notag\\
%\frac{D(a)}{a(1-a)}\langle a\rangle_2&=-\frac{D(1-a)}{(1-a)\{1-(1-a)\}}\langle 1-a\rangle_2\notag\\
%\llbracket a\rrbracket^D_2&=-\llbracket 1-a\rrbracket^D_2\notag
%\end{align}
2. Inversion relation: 
\[\llbracket a\rrbracket^D_2=-\left\llbracket \frac{1}{a}\right\rrbracket^D_2\]
% We are reducing calculations here
%The inversion relation in $\beta_2(F)$ is
%\begin{align}
%\left \langle a\right\rangle_2=&-a\left\langle \frac{1}{a}\right\rangle_2\notag\\
%\frac{D(a)}{a(1-a)}\left\langle a\right\rangle_2=&\frac{D(a)}{a(1-a)}\cdot-a\left\langle\frac{1}{a}\right\rangle_2\notag\\
%\frac{D(a)}{a(1-a)}\left\langle a\right\rangle_2=&\frac{\frac{1}{a^2}D(a)}{\frac{1}{a}\left(1-\frac{1}{a}\right)}\left\langle\frac{1}{a}\right\rangle_2\notag\\
%\frac{D(a)}{a(1-a)}\left\langle a\right\rangle_2=&-\frac{D\left(\frac{1}{a}\right)}{\frac{1}{a}\left(1-\frac{1}{a}\right)}\left\langle \frac{1}{a}\right\rangle_2\notag\\
%\llbracket a\rrbracket^D_2=&-\left\llbracket \frac{1}{a}\right\rrbracket^D_2\notag
%\end{align}

3. The five-term relation:
\[\llbracket a\rrbracket^D_2-\llbracket b\rrbracket^D_2+\left\llbracket \frac{b}{a}\right\rrbracket^D_2-\left\llbracket \frac{1-b}{1-a}\right\rrbracket^D_2+\left\llbracket \frac{1-b^{-1}}{1-a^{-1}}\right\rrbracket^D_2=0\]
% and $D=a(1-a)\frac{\partial}{\partial a}+b(1-b)\frac{\partial}{\partial a}$ acts as a partial derivative then the above relation is same as (\ref{catfour}). 
\begin{rem}\label{fifth}
If we use the definition of $\llbracket a\rrbracket^D_2$ for certain $D\in Der_{\varmathbb{Z}}(F)$,i.e., $D=a(1-a)\frac{\partial}{\partial a}+b(1-b)\frac{\partial}{\partial b} \in Der_{\varmathbb{Z}}(F,F)$ where $\frac{\partial}{\partial a}$ and $\frac{\partial}{\partial b}$ are the usual partial derivatives then we see that $\left\llbracket \frac{1-b^{-1}}{1-a^{-1}}\right\rrbracket^D_2=0$. This is how Cathelineau arrived at his four-term relation \end{rem}

\section{Infinitesimal complexes}\label{inf_comp}
There are some homomorphisms which relate Bloch-Suslin and Goncharov's complexes to Grassmannian complex(\cite{Gonc},\cite{Gonc1},\cite{Gonc2}). This section will relate variant of Cathelineau's infinitesimal complex to the geometric configurations of Grassmannian complex. We will suggest here some suitable maps for this relation and then will verify the commutativity of the underlying diagrams. Goncharov used $K$-theory to prove the commutativity of his diagram in which he related his complex with the Grassmannian complex (see appendix of \cite{Gonc4}) but here we are giving proof of the commutativity of diagram (\ref{bicomp2})(see $\S$\ref{add_tri} below) without using $K$-theory we shall use combinatorial techniques with the rewriting of triple ratio into a product of two cross-ratios.
%The same technique can also be used in Goncharov's case (see appendix \ref{Gonc_app}).

Throughout this section we will work with modulo 2-torsion and use $D\in Der_{\varmathbb{Z}}F$ as an absolute derivation for a field $F$. For $\S$\ref{add_dil} determinant $\Delta$ is defined as $\Delta(l_i,l_j):=\langle \omega,l_i\wedge l_j\rangle$, for $l_i,l_j\in V_2$, where $\omega \in \det V^*_2$ is volume form in $V_2$. For $\S$\ref{add_tri} determinant $\Delta$ is defined as $\Delta(l_i,l_j,l_k):=\langle \omega,l_i\wedge l_j\wedge l_k\rangle$ for $l_i,l_j,l_k\in V_3$, where $\omega \in \det V^*_3$ is volume form in $V_3$.
\subsection{Infinitesimal Dilogarithm}\label{add_dil}
Let $C_m(2)$ (or $C_m(\varmathbb{P}^1_F)$) be the free abelian group generated by the configurations of $m$ vectors in a two dimensional vector space $V_2$ over a field $F$ (or $m$ points in $\varmathbb{P}_F^1$) in generic position. Configurations of $m$ vectors in vector space $V_2$ are 2-tuples of vectors modulo GL$_2(V_2)$-equivalence. Grassmannian subcomplex (see diagram \ref{Gonbicomp} in $\S$\ref{Gra_comp}) for this case is the following.
\begin{equation}
\cdots\xrightarrow{d}C_5(2)\xrightarrow{d}C_4(2)\xrightarrow{d}C_3(2) \notag
\end{equation}
\begin{equation}
d\colon(l_0,\ldots,l_{m-1})\mapsto\sum_{i=0}^m (-1)^i(l_0,\ldots,\hat{l_i},\ldots,l_{m-1})\notag
\end{equation}
We will outline the procedure initially for $V_2$ and then will proceed further for $V_3$. We will also use the process of derivation (see $\S$\ref{def1}) in combination with cross-ratio to define our maps. 

Consider the following diagram
\begin{displaymath}\label{bicomplex1}
\xymatrix{
C_5(2)\ar[r]^{d}      & C_4(2)\ar[r]^d\ar[d]^{\tau_1^{2}}        & C_3(2)\ar[d]^{\tau_0^{2}}\\
                     & \beta^D_2(F)\ar[r]^{\partial^D}              & F\otimes F^\times }\tag{3.1a}\
\end{displaymath}
where $\beta^D_2(F)$ and $\partial^D$ are defined in $\S$\ref{def1}, we define 
\begin{align}\label{todef}
\tau_0^{2}\colon(l_0,l_1,l_2)\mapsto &\sum_{i=0}^2 \frac{D\{\Delta(l_i,l_{i+2})\}}{\Delta(l_i,l_{i+2})}\otimes \Delta(l_i,l_{i+1})\notag\\&\quad-\frac{D\{\Delta(l_{i+1},l_i)\}}{\Delta(l_{i+1},l_i)}\otimes \Delta(l_i,l_{i+2})\}\quad \text{      $i$ mod 3}
\end{align}
Note: The above can also be written as:
\[\sum_{i=0}^2 \frac{D\{\Delta(l_i,l_{i+2})\}}{\Delta(l_i,l_{i+2})}\otimes\frac{\Delta(l_i,l_{i+1})}{\Delta(l_{i-1},l_{i+1})},\hspace{2pt}i\text{     mod } 3.\]
Furthermore, we put  
\begin{equation}\label{t1def}
\tau_1^{2}\colon (l_0,\ldots,l_3)\mapsto \llbracket r(l_0,\ldots,l_3)\rrbracket^D_2\notag
\end{equation}
where $\llbracket a \rrbracket^D_2=\frac{D(a)}{a(1-a)}\langle a\rangle$ (defined in $\S$\ref{def1}) and  $r(l_0,\ldots,l_3)=\frac{\Delta(l_0,l_3)\Delta(l_1,l_2)}{\Delta(l_0,l_2)\Delta(l_1,l_3)}$ is the cross ratio of the points $(l_0,\ldots,l_3) \in C_4(\varmathbb{P}^1_F)$(defined in $\S$\ref{b_s}) and following is the cross-ratio identity for $(l_0,\ldots,l_3)\in C_4(2)$.
\begin{equation}\label{2did}
\Delta(l_0,l_1)\Delta(l_2,l_3)=\Delta(l_0,l_2)\Delta(l_1,l_3)-\Delta(l_0,l_3)\Delta(l_1,l_2)
\end{equation}

To ensure well-definedness of our homomorphisms $\tau_0^2$ and $\tau_1^2$ above, we first show that the definition is independent of length of the vectors and volume form $\omega$. Here are some results for verification.
\begin{lem}
$\tau_0^{2}$ is independent of the volume form $\omega$ by the vectors in $V_2$.
\end{lem}
\begin{proof}
According to (\ref{todef}), $\tau_0^{2}$ can be written for the vectors $(l_0,l_1,l_2)$ as
\begin{align}
\tau_0^{2}(l_0,l_1,l_2) &=\frac{D\{\Delta(l_0,l_2)\}}{\Delta(l_0,l_2)}\otimes\Delta(l_0,l_1)-\frac{D\{\Delta(l_0,l_1)\}}{\Delta(l_0,l_1)}\otimes\Delta(l_0,l_2)\notag\\
&+\frac{D\{\Delta(l_1,l_0)\}}{\Delta(l_1,l_0)}\otimes\Delta(l_1,l_2)-\frac{D\{\Delta(l_1,l_2)\}}{\Delta(l_1,l_2)}\otimes\Delta(l_1,l_0)\notag\\
&+\frac{D\{\Delta(l_2,l_1)\}}{\Delta(l_2,l_1)}\otimes\Delta(l_2,l_0)-\frac{D\{\Delta(l_2,l_0)\}}{\Delta(l_2,l_0)}\otimes\Delta(l_2,l_1) \notag 
\end{align}
further we can also write as
\begin{equation}
\tau_0^{2}(l_0,l_1,l_2)=\frac{D\{\Delta(l_0,l_2)\}}{\Delta(l_0,l_2)}\otimes\frac{\Delta(l_0,l_1)}{\Delta(l_2,l_1)}-\frac{D\{\Delta(l_0,l_1)\}}{\Delta(l_0,l_1)}\otimes\frac{\Delta(l_0,l_2)}{\Delta(l_1,l_2)}+\frac{D\{\Delta(l_1,l_2)\}}{\Delta(l_1,l_2)}\otimes\frac{\Delta(l_2,l_0)}{\Delta(l_1,l_0)}\notag
\end{equation}
Changing the volume form $\omega\mapsto\lambda\omega$ does not change the expression on RHS, due to homogeneity of the terms of the RH factors.
%If we use the definition of $\Delta$ in terms of $\omega$ in the above equation then the result will remain same.
\end{proof}

Next lemma will show independence of the length of the vectors. 
\begin{lem}
$\tau_0^{2}\circ d(l_0,\ldots,l_3)$ does not depend on the length of the vectors $l_i$ in $V_2$.
\end{lem}
\begin{proof}
By using a simple calculation we can write
\begin{align} \label{tod}
\tau_0^{2}\circ d(l_0,\ldots,l_3)
&=\frac{D\left\lbrace\Delta(l_0,l_1)\Delta(l_2,l_3)\right\rbrace}{\Delta(l_0,l_1)\Delta(l_2,l_3)}\otimes\frac{\Delta(l_0,l_2)\Delta(l_1,l_3)}{\Delta(l_0,l_3)\Delta(l_1,l_2)}\notag \\
&-\frac{D\left\lbrace\Delta(l_1,l_2)\Delta(l_0,l_3)\right\rbrace}{\Delta(l_1,l_2)\Delta(l_0,l_3)}\otimes\frac{\Delta(l_1,l_3)\Delta(l_0,l_2)}{\Delta(l_0,l_1)\Delta(l_2,l_3)} \notag \\
&+\frac{D\left\lbrace\Delta(l_0,l_2)\Delta(l_1,l_3)\right\rbrace}{\Delta(l_0,l_2)\Delta(l_1,l_3)}\otimes\frac{\Delta(l_0,l_3)\Delta(l_2,l_1)}{\Delta(l_0,l_1)\Delta(l_2,l_3)}\
\end{align}
now consider $\lambda \in F^\times$ and we know that $\frac{D(\lambda x)}{\lambda x}= \frac{D(x)}{x}$ for $\lambda \in F^\times$ and the other part of the right hand side is a cross-ratio.
\end{proof}

Note: Since $\tau_1^2$ is defined via cross-ratio and $d\log$ so there is no need to check things that are mandatory for $\tau_0^2$.
\begin{prop}\label{claim1}
The diagram below is commutative.
\end{prop}
\begin{displaymath}\label{bicompb}
 \xymatrix{
	& C_4(2)\ar[r]^d\ar[d]^{\tau_1^{2}}        & C_3(2)\ar[d]^{\tau_0^{2}}\\
                     & \beta^D_2(F)\ar[r]^{\partial^D}               & F\otimes F^\times }\tag{3.1b}\\
\end{displaymath}
\begin{proof}
The first thing is to calculate $\partial^D\circ \tau_1^{2}(l_0,\ldots,l_3)$ because we have already computed $\tau_0^{2}\circ d(l_0,\ldots,l_3)$ in (\ref{tod}) then by (\ref{t1def})
\[\tau_1^{2}(l_0,\cdots,l_3)= \left\llbracket\frac{\Delta(l_0,l_3)\Delta(l_1,l_2)}{\Delta(l_0,l_2)\Delta(l_1,l_3)}\right\rrbracket^D_2\]

According to this we can identify $l_0,\ldots,l_3$ with points in $\varmathbb{P}_F^1$, then by the 3-fold transitivity of PGL$_2(F)$ any  $(l_0,\ldots,l_3)\in (\varmathbb{P}_F^1)^4$ in generic position is PGL$_2(F)$ equivalent to $(0,\infty,1,a)$ for some $a\in F$
\[\tau_1^{2}\left(0,\infty,1,a\right)=\frac{D(a)}{a(1-a)}\langle a\rangle_2=\llbracket a\rrbracket^D_2 \text{     for any } a \in \varmathbb{P}_F^1-\{0,1,\infty\}\]
where $D\log(a)=\frac{D(a)}{a}$. Calculate $\partial^D\big(\llbracket a\rrbracket^D_2\big)$
\begin{align}
&=-\frac{D(1-a)}{(1-a)}\otimes a+\frac{D(a)}{a}\otimes (1-a)\notag\
\end{align}
For the vectors in $C_4(2)$ and by using the identity (\ref{2did}) we can write
\begin{align}
\partial^D\circ \tau_1^{2}(l_0,\ldots,l_3)=&-\frac{D\left\lbrace\frac{\Delta(l_0,l_1)\Delta(l_2,l_3)}{\Delta(l_0,l_2)\Delta(l_1,l_3)}\right\rbrace}{\frac{\Delta(l_0,l_1)\Delta(l_2,l_3)}{\Delta(l_0,l_2)\Delta(l_1,l_3)}}\otimes\frac{\Delta(l_0,l_3)\Delta(l_1,l_2)}{\Delta(l_0,l_2)\Delta(l_1,l_3)}\notag\\
&+\frac{D\left\lbrace\frac{\Delta(l_0,l_3)\Delta(l_1,l_2)}{\Delta(l_0,l_2)\Delta(l_1,l_3)}\right\rbrace}{\frac{\Delta(l_0,l_3)\Delta(l_1,l_2)}{\Delta(l_0,l_2)\Delta(l_1,l_3)}}\otimes\frac{\Delta(l_0,l_1)\Delta(l_2,l_3)}{\Delta(l_0,l_2)\Delta(l_1,l_3)}\notag
\end{align}
by using $\frac{D\left(\frac{a}{b}\right)}{\left(\frac{a}{b}\right)}=\frac{D(a)}{a}-\frac{D(b)}{b}$ and then cancelling two terms  we can convert the above into (\ref{tod}) and the diagram (\ref{bicompb}) is commutative.
\end{proof}
Further consider the diagram (\ref{bicomplex1}) and note that $\tau_1^{2}\circ d$ becomes
\begin{equation}\label{t1od}
\tau_1^{2}\circ d(l_0,\ldots,l_4)=\sum_{i=0}^4(-1)^i\llbracket r(l_0,\ldots,\hat{l_i},\ldots,l_4)\rrbracket^D_2\notag
\end{equation}
Now we can further verify that $\tau_1^{2}\circ d(l_0,\ldots,l_4) \in \ker (\partial^D)$
\begin{align}
&\partial^D\circ(\tau_1^{2}\circ d(l_0,\ldots,l_4))\notag\\
&=\sum_{i=0}^4\Big(
-\frac{D\left\lbrace1-r(l_0,\ldots,\hat{l_i},\ldots,l_4)\right\rbrace}{1-r(l_0,\ldots,\hat{l_i},\ldots,l_4)}\otimes r(l_0,\ldots,\hat{l_i},\ldots,l_4)\notag\\
&\quad\quad\quad+\frac{D\left\lbrace r(l_0,\ldots,\hat{l_i},\ldots,l_4)\right\rbrace}{r(l_0,\ldots,\hat{l_i},\ldots,l_4)}\otimes\left\lbrace1-r(l_0,\ldots,\hat{l_i},\ldots,l_4)\right\rbrace\Big)\notag
\end{align}
From now on we will write $(ij)$ for $\Delta(l_i,l_j)$ in short. The above expression can also be written for each value of $i$'s, e.g.

\[\text{ for $i=0$ we have }-\frac{D\left\lbrace\frac{(12)(43)}{(13)(42)}\right\rbrace}{\frac{(12)(43)}{(13)(42)}}\otimes\frac{(14)(23)}{(13)(24)}+\frac{D\left\lbrace\frac{(14)(23)}{(13)(24)}\right\rbrace}{\frac{(14)(23)}{(13)(42)}}\otimes\frac{(12)(43)}{(13)(42)}\]
and similarly for others as well.
%\notag\\&-\frac{D\left\lbrace\frac{(02)(43)}{(03)(42)}\right\rbrace}{\frac{(02)(43)}{(03)(42)}}\otimes\frac{(04)(23)}{(03)(24)}+\frac{D\left\lbrace\frac{(04)(23)}{(03)(24)}\right\rbrace}{\frac{(04)(23)}{(03)(24)}}\otimes\frac{(02)(43)}{(03)(42)}\notag\\&+\frac{D\left\lbrace\frac{(01)(43)}{(03)(41)}\right\rbrace}{\frac{(01)(43)}{(03)(41)}}\otimes\frac{(04)(13)}{(03)(14)}-\frac{D\left\lbrace\frac{(04)(13)}{(03)(14)}\right\rbrace}{\frac{(04)(13)}{(03)(14)}}\otimes\frac{(01)(43)}{(03)(41)}\notag\\&-\frac{D\left\lbrace\frac{(01)(42)}{(02)(41)}\right\rbrace}{\frac{(01)(42)}{(02)(41)}}\otimes\frac{(04)(12)}{(02)14)}+\frac{D\left\lbrace\frac{(04)(12)}{(02)(14)}\right\rbrace}{\frac{(04)(12)}{(02)(14)}}\otimes\frac{(01)(42)}{(02)(41)}\notag\\&+\frac{D\left\lbrace\frac{(01)(32)}{(02)(31)}\right\rbrace}{\frac{(01)(32)}{(02)(31)}}\otimes\frac{(03)(12)}{(02)(13)}-\frac{D\left\lbrace\frac{(03)(12)}{(02)(13)}\right\rbrace}{\frac{(03)(12)}{(02)(13)}}\otimes\frac{(01)(32)}{(02)(31)}\notag

If we multiply out, using $\frac{D(ab)}{ab}=\frac{D(a)}{a}+\frac{D(b)}{b}$ and start to collect each term of the form $\frac{D(ij)}{(ij)}\otimes \cdots$ from the above i.e. fix $i$ and $j$, calculate the sum of all, then we will be able to see that every individual term of $\frac{D(ij)}{(ij)}\otimes \cdots$ is 0. For example $\frac{D(01)}{(01)}\otimes\frac{(04)(13)}{(03)(14)}\frac{(02)(14)}{(04)(12)}\frac{(03)(12)}{(02)(13)}=0$ since the RHS is 2-torsion in $F^\times$ so we can easily say that the above is zero and $\tau_1^{2}\circ d\in \ker (\partial^D)$.

\textbf{Projeced cross-ratio:} For $l_0,\ldots,l_4\in \varmathbb{P}^2_F$, $r(l_0|l_1,l_2,l_3,l_4)$ is the projected cross-ratio of four points $l_0,\ldots,l_4$ projeced from $l_0$ and is defined as 
\[r(l_0|l_1,l_2,l_3,l_4)=\frac{\Delta(l_0,l_1,l_4)\Delta(l_0,l_2,l_3)}{\Delta(l_0,l_1,l_3)\Delta(l_0,l_2,l_4)},\]
where $\Delta(l_i,l_j,l_k)$ is a $3\times3$ determinant for $l_i,l_j,l_k\in\varmathbb{P}^2_F$
\begin{lem}\label{Gon5term}
(Goncharov, A. B., \cite{Gonc}) Let $x_0,\ldots,x_4$ be five points in generic position in $\varmathbb{P}^2_F$. Then
\[\sum_{i=0}^4(-1)^i[r(x_i|x_0,\ldots,\hat{x}_i,\ldots,x_4)]=0{} \in \mathcal{B}_2(F),\]
where $r(x_0|x_1,x_2,x_3,x_4)$ is the projected cross-ratio of four points $x_1,\ldots,x_4$ projected from $x_0$

\end{lem}
See Lemma 2.18 in \cite{Gonc} for the proof.\hfill $\Box$

In continuation of the above lemma we have a similar result here which shows that the projected five-term (or four-term in special condition) relation can also be presented for $\beta_2^D(F)$ in the same way using geometric configurations of five points in $\varmathbb{P}^2_F$.
\begin{lem}\label{4pt}
Let $x_0,\ldots,x_4$ be 5 points in generic position in $\varmathbb{P}_F^2$ then, for any $D\in Der_{\varmathbb{Z}}F$ 
\begin{align}\label{4term5}
\sum_{i=0}^4(-1)^i\llbracket r(x_i|x_0,\ldots,\hat{x_i},\ldots,x_4)\rrbracket^D_2=0 \in \beta^D_2(F)
\end{align}
\end{lem}
\begin{proof}
If $x_0,\ldots,x_4$ in $\varmathbb{P}_F^2$ then Lemma \ref{Gon5term} gives projected five-term relation  
\[\sum_{i=0}^4(-1)^i[r(x_i|x_0,\ldots,\hat{x_i},\ldots,x_4)]=0 \in B_2(F).\]
According to definition of $D\in Der_{\varmathbb{Z}}F$ in $\S$\ref{def1} above \eqref{4term5} is the five-term relation in $\beta^D_2(F)$.
\end{proof}
%As we discussed the definition of $D$ in section (\ref{def1}) and also ``derivation map'' is used from polylogarithmic group to infinitesimal polylogarithmic group (see $\S$6 of \cite{PandG}) we can say that (\ref{4term5}) is a five-term relation in $\beta^D_2(F)$.\hfill $\Box$
\begin{ex}\label{four_term_ex}
\end{ex}
By appropriate  specialization  of the configuration in $C_5(2)$, we can use (*) to produce Cathelineau's four-term relation from the geometric configurations by using the operator $D=a(1-a)\frac{\partial}{\partial a}+b(1-b)\frac{\partial}{\partial b}$ for $F=K(a,b)$ where $a$ and $b$ are indeterminates  over the field $K$ and $\frac{\partial}{\partial a}$ and $\frac{\partial}{\partial b}$ are the usual partial derivatives (see $\S$6 of \cite{PandG}). Let $(0,\infty,1,a,b)\in (\varmathbb{P}_F^1)^5$ in generic position be the five-tuple corresponding to $\left(l_0,\ldots,l_4\right)=\left(\left(\begin{array}{c}0\\1\end{array}\right),\left(\begin{array}{c}1\\0\end{array}\right),\left(\begin{array}{c}1\\1\end{array}\right),\left(\begin{array}{c}a\\1\end{array}\right),\left(\begin{array}{c}b\\1\end{array}\right)\right)\in C_5(2)$. Calculate all possible determinants formed by $(l_0,\ldots,l_4)\in C_5(2)$, i.e. $\Delta(l_i,l_j)\text{ for }0\leq i<j\leq4$, put all of them in (\ref{t1od}), we get
\[\llbracket a \rrbracket^D_2-\llbracket b \rrbracket^D_2+\left\llbracket \frac{b}{a}\right\rrbracket^D_2-\left\llbracket \frac{1-b}{1-a}\right\rrbracket^D_2+\left\llbracket \frac{1-\frac{1}{b}}{1-\frac{1}{a}}\right\rrbracket^D_2=0\]
since $\tau^2_1\circ d\in\ker(\partial^D)$, then we use $D$ defined above, calculate each term of the above to form cathelineau's four-term relation:
\begin{equation} \label{fterm}
\langle a\rangle-\langle b\rangle+a\left\langle \frac{b}{a}\right\rangle+(1-a)\left\langle \frac{1-b}{1-a}\right\rangle=0\notag
\end{equation}

\subsection{Infinitesimal Trilogarithm}\label{add_tri}
Let $C_m(3)$ (or $C_m(\varmathbb{P}^2_F)$) be the free abelian group generated by the configurations of $m$ vectors in a three dimensional vector space $V_3$ over a field $F$ (or $m$ points in $\varmathbb{P}_F^2$) in generic position.
Consider the following diagram

\begin{displaymath}\label{bicomp2}
\xymatrix{
C_6(3)\ar[r]^{d}\ar[d]^{\tau_2^3}     & C_5(3)\ar[r]^{d}\ar[d]^{\tau_1^{3}}        &C_4(3)\ar[d]^{\tau_0^{3}}\\
\beta^D_3(F)\ar[r]^{\partial\qquad\qquad}     & (\beta^D_2(F)\otimes F^\times) \oplus  (F\otimes \mathcal{B}_2(F))\ar[r]^{\qquad\qquad\partial}          & F\otimes \bigwedge^2 F^\times}\ \tag{3.2a}
\end{displaymath}
where 
\begin{align} \label{t03def}
\tau_0^{3}:(l_0,\ldots,l_3)\mapsto\sum_{i=0}^3(-1)^{i}&\frac{D\Delta(l_0,\ldots,\hat{l_i},\ldots,l_3)}{\Delta(l_0,\ldots,\hat{l_i},\ldots,l_3)}\otimes \frac{\Delta(l_0,\ldots,\hat{l}_{i+1},\ldots,l_3)}{\Delta(l_0,\ldots,\hat{l}_{i+2},\ldots,l_3)}\notag\\ &\wedge\frac{\Delta(l_0,\ldots,\hat{l}_{i+3},\ldots,l_3)}{\Delta(l_0,\ldots,\hat{l}_{i+2},\ldots,l_3)} 
\end{align}
\begin{align}
\tau_1^{3}:(l_0,\ldots,l_4)\mapsto -\frac{1}{3}\sum_{i=0}^4 &(-1)^i\{
\llbracket r(l_i|l_0,\ldots,\hat{l}_i,\ldots,l_4) \rrbracket^D_2 \otimes \prod_{j\neq i}\Delta(\hat{l}_i,\hat{l}_j) \notag \\
&+\frac{D\left(\prod_{j\neq i}\Delta(\hat{l}_i,\hat{l}_j)\right)}{\prod_{j\neq i}\Delta(\hat{l}_i,\hat{l}_j)}\otimes [r(l_i|l_0,\ldots,\hat{l}_i,\ldots,l_4)]_2\} \notag
\end{align}
\begin{align}
&\tau_2^3:(l_0,\ldots,l_5)\mapsto\frac{2}{45}\text{Alt}_6\left\llbracket \frac{\Delta(l_0,l_1,l_3)\Delta(l_1,l_2,l_4)\Delta(l_2,l_0,l_5)}{\Delta(l_0,l_1,l_4)\Delta(l_1,l_2,l_5)\Delta(l_2,l_0,l_3)}\right\rrbracket^D_3\notag
\end{align}
where \[\llbracket a\rrbracket^D_3 = \frac{D(a)}{a(1-a)}\langle a\rangle_3\text{ and } \Delta(\hat{l}_i,\hat{l}_j)=\Delta(l_0,\ldots,\hat{l}_i,\ldots,\hat{l}_j,\ldots,l_4)\]
\[\partial^D\left(\left\llbracket a\right\rrbracket^D_3\right) =\left\llbracket a \right\rrbracket^D_2\otimes a +\frac{D(a)}{a}\otimes [a]_2\]
\[\partial^D\left(\llbracket a\rrbracket^D_2\otimes b + x\otimes [y]_2\right)=\frac{D(1-a)}{1-a}\otimes a\wedge b-\frac{D(a)}{a}\otimes (1-a)\wedge b+x\otimes (1-y)\wedge y\]
First we need to show that our maps $\tau^3_0$ and $\tau^3_1$ are independent of the chosen volume form $\omega$. There is no need to show that same thing for the map $\tau^3_2$. The proofs of the following three lemmas are similar to those in $\S$3 of \cite{Gonc}. 
\begin{lem}
$\tau_0^{3}$ is independent of the volume element $\omega \in \det V^*_3$.
\end{lem}
\begin{proof}
We can write equation (\ref{t03def}) in the form
\begin{align}
\tau_0^{3}(l_0,\ldots,l_3)=\sum_{i=0}^3(-1)^{i+1}&\frac{D\Delta(l_0,\ldots,\hat{l_i},\ldots,l_3)}{\Delta(l_0,\ldots,\hat{l_i},\ldots,l_3)}\otimes \frac{\Delta(l_0,\ldots,\hat{l}_{i+1},\ldots,l_3)}{\Delta(l_0,\ldots,\hat{l}_{i+2},\ldots,l_3)}\notag\\ &\wedge\frac{\Delta(l_0,\ldots,\hat{l}_{i+2},\ldots,l_3)}{\Delta(l_0,\ldots,\hat{l}_{i+3},\ldots,l_3)}
\end{align}
If we apply the definition of $\Delta$ in terms of $\omega$ in the above then the last two factors will remain unchanged and we know that $\frac{D(\lambda a)}{\lambda a}=\frac{D(a)}{a}$ for all $\lambda \in F^\times$.
\end{proof}
\begin{lem}
$\tau_1^{3}$ is independent of the volume element $\omega \in \det V_3^*$.
\end{lem}
\begin{proof}
To prove the above we will take the difference of the elements $\tau_1^{3}(l_0,\ldots,l_4)$ by using the volume forms $\lambda\cdot\omega$ and $\omega(\lambda \in F^\times)$, term of  type $F\otimes\mathcal{B}_2(F)$ will be zero while the term of type $\beta_2^D(F)\otimes F^\times$ will be 
\begin{align}
=-\frac{1}{3}\sum_{i=0}^4 &(-1)^i\Big(\llbracket r(l_i|l_0,\ldots,\hat{l}_i,\ldots,l_4) \rrbracket^D_2 \otimes\lambda^4\prod_{i\neq j}\Delta(\hat{l}_i,\hat{l}_j)\notag\\
&-\llbracket r(l_i|l_0,\ldots,\hat{l}_i,\ldots,l_4) \rrbracket^D_2\otimes\prod_{i\neq j}\Delta(\hat{l}_i,\hat{l}_j)\notag\\
=-\frac{1}{3}\sum_{i=0}^4&(-1)^i\llbracket r(l_i|l_0,\ldots,\hat{l}_i,\ldots,l_4) \rrbracket^D_2 \otimes\lambda^4\notag
\end{align}
We use lemma \ref{4pt} which shows that left factor of the above is simply the projected five-term relation in $\beta^D_2(F)$. 
\end{proof}

Now we need to show here that the composition map $\tau^3_1\circ d$ is independent of the length of the vectors in $V_3$.
\begin{lem}
$\tau_1^{3}\circ d$ does not depend on the length of the vectors $l_i$ in $V_3$.
\end{lem}
\begin{proof}
The proof of this lemma is quite similar to the proof of proposition 3.9 of \cite{Gonc}, but we will out line here main steps because this proof involves more calculations. It is enough to prove that the following
\[\tau_1^{(3)}\circ d\{(l_0,\ldots,l_5)-(\lambda_0l_0,\ldots,\lambda_5l_5)\}=0\quad (\lambda_i \in F^\times) \]
We will consider the case $\lambda_1=\cdots=\lambda_5=1$ and $\lambda_0=\lambda$

The first summand $(l_1,\ldots,l_5)$ will not give any contribution to the difference
\begin{align}\label{eq1}
\tau_1^{3}\circ d\{(l_0,\ldots,l_5)-(\lambda_0l_0,\ldots,l_5)\}
\end{align}
Now consider the second summand $-(l_0,l_2,l_3,l_4,l_5)$
\begin{align*}
\frac{1}{3}\Bigg(-&\llbracket r(l_0|l_2,l_3,l_4,l_5) \rrbracket^D_2\otimes\prod_{j=2}^5\Delta(\hat{l}_0,\hat{l}_j)+\llbracket r(l_2|l_0,l_3,l_4,l_5)\rrbracket^D_2\otimes \lambda^3\prod_{j=0,3,4,5}\Delta(\hat{l}_2,\hat{l}_j)\\
-&\llbracket r(l_3|l_0,l_2,l_4,l_5)\rrbracket^D_2\otimes \lambda^3\prod_{j=0,2,4,5}\Delta(\hat{l}_3,\hat{l}_j)+\llbracket r(l_4|l_0,l_2,l_3,l_5)\rrbracket^D_2\otimes \lambda^3\prod_{j=0,2,3,5}\Delta(\hat{l}_4,\hat{l}_j)\\
-&\llbracket r(l_5|l_0,l_2,l_3,l_4)\rrbracket^D_2\otimes \lambda^3\prod_{j=0,2,3,4}\Delta(\hat{l}_5,\hat{l}_j)\\
+&\sum^5_{\substack{i=0\\i\neq1}}\frac{D\left(\prod_{j\neq1,i}\Delta(\hat{l}_i,\hat{l}_j)\right)}{\Delta(\hat{l}_i,\hat{l}_j)}\otimes \left[r(l_i|l_0,\ldots,\hat{l}_i,\ldots,l_4)\right]\Bigg)\\
-\frac{1}{3}\Bigg(-&\llbracket r(l_0|l_2,l_3,l_4,l_5) \rrbracket^D_2\otimes\prod_{j=2}^5\Delta(\hat{l}_0,\hat{l}_j)+\llbracket r(l_2|l_0,l_3,l_4,l_5)\rrbracket^D_2\otimes \prod_{j=0,3,4,5}\Delta(\hat{l}_2,\hat{l}_j)\\
-&\llbracket r(l_3|l_0,l_2,l_4,l_5)\rrbracket^D_2\otimes \prod_{j=0,2,4,5}\Delta(\hat{l}_3,\hat{l}_j)+\llbracket r(l_4|l_0,l_2,l_3,l_5)\rrbracket^D_2\otimes \prod_{j=0,2,3,5}\Delta(\hat{l}_4,\hat{l}_j)\\
-&\llbracket r(l_5|l_0,l_2,l_3,l_4)\rrbracket^D_2\otimes \prod_{j=0,2,3,4}\Delta(\hat{l}_5,\hat{l}_j)\\
+&\sum^5_{\substack{i=0\\i\neq1}}\frac{D\left(\prod_{j\neq1,i}\Delta(\hat{l}_i,\hat{l}_j)\right)}{\Delta(\hat{l}_i,\hat{l}_j)}\otimes \left[r(l_i|l_0,\ldots,\hat{l}_i,\ldots,l_4)\right]\Bigg)
\end{align*}
This difference gives us
\begin{align}\label{eq2}
&\frac{1}{3}\Big(\llbracket r(l_2|l_0,l_3,l_4,l_5) \rrbracket^D_2-\llbracket r(l_3|l_0,l_2,l_4,l_5)\rrbracket^D_2\notag\\
&+\llbracket r(l_4|l_0,l_2,l_3,l_5)\rrbracket^D_2-\llbracket r(l_5|l_0,l_2,l_3,l_4)\rrbracket^D_2 \Big)\otimes \lambda^3 
\end{align}
If we apply lemma \ref{4pt} to the 5-tuple $(l_0,l_2,l_3,l_4,l_5)$ of points in $\varmathbb{P}_F^2$ then we see that
\begin{align}
\llbracket r(l_0|l_2,l_3,l_4,l_5) \rrbracket^D_2&=\llbracket r(l_2|l_0,l_3,l_4,l_5) \rrbracket^D_2-\llbracket r(l_3|l_0,l_2,l_4,l_5) \rrbracket^D_2 \notag\\
&+\llbracket r(l_4|l_0,l_2,l_3,l_5)\rrbracket^D_2-\llbracket r(l_5|l_0,l_2,l_3,l_4)\rrbracket^D_2 \notag 
\end{align}
Then equation \ref{eq2} can be written as 
\begin{equation}\label{eq3}
\frac{1}{3}\llbracket r(l_0|l_2,l_3,l_4,l_5) \rrbracket^D_2\otimes \lambda^3
\end{equation}
The contribution of the summand $(-1)^i(l_0,\ldots,\hat{l}_i,\ldots,l_5)$ in equation \ref{eq3} is
\[\frac{1}{3}(-1)^{i-1}\llbracket r(l_0|l_1,\ldots,\hat{l}_i,\ldots,l_5)\rrbracket^D_2\otimes \lambda^3\]
Now for all summands
\begin{align}
\frac{1}{3}\sum_{i=1}^5(-1)^{i-1}\llbracket r(l_0|l_1,\ldots,\hat{l}_i,\ldots,l_5)\rrbracket^D_2\otimes \lambda^3 \notag
\end{align}
According to lemma \ref{4pt} left factor of the above is projected five-term relation in $\beta^D_2(F)$ and is zero.
\end{proof}
\begin{thm}\label{claim3a}
The following diagram 
\begin{displaymath}
\xymatrix{
C_5(3)\ar[d]^{\tau^3_1}\ar[rr]^{d}&&C_4(3)\ar[d]^{\tau^3_0}\\
\left(\beta^D_2(F)\otimes F^\times\right)\oplus\left(F\otimes \mathcal{B}_2(F)\right)\ar[rr]^{\qquad\qquad\partial}&&F\otimes \bigwedge^2F^\times}
\end{displaymath}

is commutative i.e. $\tau_0^{3}\circ d=\partial^D\circ\tau_1^{3}$
\end{thm}
\begin{proof}
From now on we will denote $\Delta(l_0,l_1,l_2)=(l_0,l_1,l_2)$
\begin{align}\label{lhs3}
&\tau_0^{3}\circ d(l_0,\ldots,l_4)\notag\\
&=\tau_0^{3}\left(\sum_{i=0}^4(-1)^i(l_0,\ldots,\hat{l}_i,\ldots,l_4)\right)\notag \\
&=\widetilde{\text{Alt}}_{(01234)}\Bigg(\sum^3_{i=0}(-1)^i\frac{D(l_0,\ldots,\hat{l}_i,\ldots,\hat{l}_3)}{(l_0,\ldots,\hat{l}_i,\ldots,\hat{l}_3)}\otimes\frac{(l_0,\ldots,\hat{l}_{i+1},\ldots,\hat{l}_3)}{(l_0,\ldots,\hat{l}_{i+2},\ldots,\hat{l}_3)}\notag\\
&\quad\quad\quad\quad\quad\quad\quad\wedge\frac{(l_0,\ldots,\hat{l}_{i+3},\ldots,\hat{l}_3)}{(l_0,\ldots,\hat{l}_{i+2},\ldots,\hat{l}_3)}\Bigg),\qquad i\mod 4 
\end{align}
where $\widetilde{\text{Alt}}$ differs from usual alternation sum in the sense that we do not divide by the order of the group for $\widetilde{\text{Alt}}$. If we expand the inner sum first then we will get 4 terms which can be simplified in 12 terms, i.e., we will have terms of the following shape:
\[\frac{D(l_1,l_2,l_3)}{(l_1,l_2,l_3)}\otimes(l_0,l_2,l_3)\wedge(l_0,l_1,l_3)\quad\text{and so on}\]
Then we pass to the alternation which gives us 60 terms so we keep together those terms which have same first factor e.g., 
\begin{align}
+&\frac{D(l_0,l_1,l_2)}{(l_0,l_1,l_2)}\otimes\{(l_0,l_1,l_3)\wedge(l_1,l_2,l_3)-(l_0,l_1,l_4)\wedge(l_1,l_2,l_4)-(l_0,l_2,l_3)\wedge(l_1,l_2,l_3)\notag\\
&\quad\quad\quad\quad\quad+(l_0,l_2,l_4)\wedge(l_1,l_2,l_4)-(l_0,l_1,l_3)\wedge(l_0,l_2,l_3)+(l_0,l_1,l_4)\wedge(l_0,l_2,l_4)\}\notag\\
&\vdots\notag\\
&\text{and so on}\notag
\end{align}

The other part of the calculation is very long and tedious but we will try to include some steps here.

Going to the other side of the diagram, we find
\begin{align}
\partial^D\circ \tau_1^{3}(l_0,\ldots,l_4)=-\frac{1}{3}\partial^D\Bigg(\sum_{i=0}^4&(-1)^i\llbracket r(l_i|l_0,\ldots,\hat{l}_i,\ldots,l_4) \rrbracket^D_2 \otimes \prod_{j\neq i}\Delta(\hat{l}_i,\hat{l}_j) \notag \\
&+\frac{D\left(\prod_{j\neq i}\Delta(\hat{l}_i,\hat{l}_j)\right)}{\prod_{j\neq i}\Delta(\hat{l}_i,\hat{l}_j)}\otimes [r(l_i|l_0,\ldots,\hat{l}_i,\ldots,l_4)]\}\Bigg) \notag
\end{align}

\begin{align}\label{inter_step}
=-\frac{1}{3}\sum_{i=0}^4&(-1)^i\Big(\frac{D\left(1-r(l_i|l_0,\ldots,\hat{l}_i,\ldots,l_4)\right)}{1-r(l_i|l_0,\ldots,\hat{l}_i,\ldots,l_4)}\otimes r(l_i|l_0,\ldots,\hat{l}_i,\ldots,l_4)\wedge \prod_{j\neq i}\Delta(\hat{l}_i,\hat{l}_j)\notag \\
&-\frac{D\left(r(l_i|l_0,\ldots,\hat{l}_i,\ldots,l_4)\right)}{r(l_i|l_0,\ldots,\hat{l}_i,\ldots,l_4)}\otimes\{1-r(l_i|l_0,\ldots,\hat{l}_i,\ldots,l_4)\}\wedge \prod_{j\neq i}\Delta(\hat{l}_i,\hat{l}_j)\notag \\
&+\frac{D\left(\prod_{j\neq i}\Delta(\hat{l}_i,\hat{l}_j)\right)}{\prod_{j\neq i}\Delta(\hat{l}_i,\hat{l}_j)}\otimes \left(1-r(l_i|l_0,\ldots,\hat{l}_i,\ldots,l_4)\right)\wedge r(l_i|l_0,\ldots,\hat{l}_i,\ldots,l_4)\Big)\notag
\end{align}

From now on we will use $(ijk)$ instead of $\Delta(l_i,l_j,l_k)$ as a shorthand. If we expand the above sum with respect to $i$, then we will get a long expression. For example when $i=0$, we have 

\begin{align}
&+\frac{D\left(\frac{(012)(034)}{(013)(024)}\right)}{\frac{(012)(034)}{(013)(024)}}\otimes\frac{(014)(023)}{(013)(024)}\wedge (234)(134)(124)(123)\notag \\
&-\frac{D\left(\frac{(014)(023)}{(013)(024)}\right)}{\frac{(014)(023)}{(013)(024)}}\otimes\frac{(012)(034)}{(013)(024)}\wedge (234)(134)(124)(123)\notag \\
&+\frac{D\left((234)(134)(124)(123)\right)}{(234)(134)(124)(123)}\otimes \frac{(012)(034)}{(013)(024)}\wedge\frac{(014)(023)}{(013)(024)}\notag
\end{align}
and we can get four more similar expressions for the other values of $i$ as well. If we collect terms of type $\frac{D(ijk)}{(ijk)}\otimes \cdots\wedge\cdots$ i.e., fix $i,j$ and $k$ in all five expressions (one of them is given above), then we will see a huge amount of terms but we cancel terms pairwise and collect terms of the same kind, we get each remaining term with the coefficient ``3''. So we can write in the following form. 
\begin{align}
-3&\frac{D(012)}{(012)}\otimes\{(013)\wedge(123)-(014)\wedge(124)-(023)\wedge(123)\notag\\
&\quad\quad\quad+(024)\wedge(124)-(013)\wedge(023)+(014)\wedge(024)\}\notag\\ 
-3&\frac{D(013)}{(013)}\otimes\{(014)\wedge(134)+(023)\wedge(123)-(012)\wedge(123)\notag\\
&\quad\quad\quad-(034)\wedge(134)-(014)\wedge(034)+(012)\wedge(023)\}\notag \\
&\vdots \notag \\
&\text{and so on}\notag 
\end{align}
It turns out that every term has ``$-3$'' as a coefficient that cancels the factor $-\frac{1}{3}$ in the definition of $\tau^3_1$ then comparing the expression above with \eqref{lhs3}, we find after a long calculation that both agree (term-wise)
%we can get the same result as we presented in (\ref{lhs3}) by placing the related terms together. Now we can say that the theorem \ref{claim2} is proved.
\end{proof}

Here we have another result which will then complete the commutativity of  diagram (\ref{bicomp2})
\begin{thm}\label{claim3b}
The following diagram 
\begin{displaymath}
\xymatrix{
C_6(3)\ar[d]^{\tau^3_2}\ar[rr]^{d}&&C_5(3)\ar[d]^{\tau^3_1}\\
\beta^D_3(F)\ar[rr]^{\partial\qquad\qquad}&&\left(\beta^D_2(F)\otimes F^\times\right)\oplus\left(F\otimes \mathcal{B}_2(F)\right)}
\end{displaymath}
is commutative i.e. $\tau^3_2\circ\partial^D=d\circ\tau^3_1$.
\end{thm}
\begin{proof}
The map $\tau^3_2$ is based on generalized cross-ratios of $3 \times 3$ determinants. The total number of terms due to map $\tau^3_2$ will be 720 which can further be reduced to 120 due to symmetry (cyclic and inverse). The direct procedure which was used in the previous proof will be very lengthy and tedious so we will use techniques of combinatorics and will rewrite the triple-ratio in to the product of two cross-ratios to prove this result. 

We first compute $\partial\circ\tau^3_2(l_0,\ldots,l_5)$ and we already have 
\[\tau_2^3(l_0,\ldots,l_5)=\frac{2}{45}\text{Alt}_6\left\llbracket \frac{\Delta(l_0,l_1,l_3)\Delta(l_1,l_2,l_4)\Delta(l_2,l_0,l_5)}{\Delta(l_0,l_1,l_4)\Delta(l_1,l_2,l_5)\Delta(l_2,l_0,l_3)}\right\rrbracket^D_3,\] 
from now, in this proof we will use $(ijk)$ for $\Delta(l_i,l_j,l_k)$ and $(0.\ldots,5)$ for $(l_0,\ldots,l_5)$ as a short hand.

The above becomes \[\tau_2^3(l_0,\ldots,l_5)=\frac{2}{45}\text{Alt}_6\left\llbracket \frac{(013)(124)(205)}{(014)(125)(203)}\right\rrbracket^D_3\]
\begin{align}\label{alt0}
\partial^D\circ\tau^3_2(l_0,\ldots,l_5)=&\frac{2}{45}\text{Alt}_6\left(\left\llbracket \frac{(013)(124)(205)}{(014)(125)(203)}\right\rrbracket^D_2\otimes \frac{(013)(124)(205)}{(014)(125)(203)}\right)\notag\\
+&\frac{2}{45}\text{Alt}_6\left(D\log\frac{(013)(124)(205)}{(014)(125)(203)}\otimes \left[\frac{(013)(124)(205)}{(014)(125)(203)}\right]_2\right)
\end{align}
First we will consider first term of the above
\begin{align*}
=\frac{2}{45}\Bigg(&\text{Alt}_6\left\lbrace\llbracket r_3(0\ldots5)\rrbracket^D_2 \otimes (013)\right\rbrace+\text{Alt}_6\left\lbrace\llbracket r_3(0\ldots5)\rrbracket^D_2 \otimes (124)\right\rbrace\notag\\
+&\text{Alt}_6\left\lbrace\llbracket r_3(0\ldots5)\rrbracket^D_2 \otimes (205)\right\rbrace-\text{Alt}_6\left\lbrace\llbracket r_3(0\ldots5)\rrbracket^D_2 \otimes (014)\right\rbrace\notag\\
-&\text{Alt}_6\left\lbrace\llbracket r_3(0\ldots5)\rrbracket^D_2 \otimes (125)\right\rbrace-\text{Alt}_6\left\lbrace\llbracket r_3(0\ldots5)\rrbracket^D_2 \otimes (203)\right\rbrace\Bigg)\notag
\end{align*}
where
\[r_3(0,\ldots,5)=\frac{(013)(124)(205)}{(014)(125)(203)}\]
Use the even cycle (012)(345)
\[\text{Alt}_6\left\lbrace\llbracket r_3(012345)\rrbracket^D_2\otimes (013)\right\rbrace=\text{Alt}_6\left\lbrace\llbracket r_3(120453)\rrbracket^D_2\otimes (124)\right\rbrace\]
Now we use $\llbracket r_3(012345)\rrbracket^D_2=\llbracket r_3(120453)\rrbracket^D_2$ and similar for the others, then the above can be written as
\begin{align*}
=\frac{2}{45}\left(3\text{Alt}_6\left\lbrace\llbracket r_3(0\ldots5)\rrbracket^D_2 \otimes (013)\right\rbrace-3\text{Alt}_6\left\lbrace\llbracket r_3(0\ldots5)\rrbracket^D_2 \otimes (014)\right\rbrace\right)
\end{align*}
Use the odd cycle (34)
\begin{align*}
=\frac{2}{45}\left(6\text{Alt}_6\left\lbrace\llbracket r_3(0\ldots5)\rrbracket^D_2 \otimes (013)\right\rbrace\right)
\end{align*}
If we apply the odd permutation (03), then
\[=\frac{2}{45}\left(3\text{Alt}_6\left\lbrace\llbracket r_3(012345)\rrbracket^D_2\otimes(013)\right\rbrace-3\text{Alt}_6\left\lbrace\llbracket r_3(312045)\rrbracket^D_2\otimes(310)\right\rbrace\right)\]
but (013)=(310) so up to 2-torsion
\[=\frac{2}{15}\text{Alt}_6\left\lbrace\Big(\llbracket r_3(012345)\rrbracket^D_2-\llbracket r_3(312045)\rrbracket^D_2\Big)\otimes(013)\right\rbrace\]
Now we will use here the crucial idea of this proof in which we will divide the triple-ratio into the product of two projected cross-ratios of four points each. There are exactly 3 ways to divide this ratio into such a product. i.e., if $r_3(a,b,c,d,e,f)$ then it can be divided by projection from $a$ and $b$, projection from $a$ and $c$ or projection from $b$ and $c$. In our case we will divide by projection from 1 and 2.
\[=\frac{2}{15}\text{Alt}_6\left\lbrace\left(\left\llbracket\frac{r(2|1053)}{r(1|0234)}\right\rrbracket^D_2-\left\llbracket\frac{r(2|1350)}{r(1|3204)}\right\rrbracket^D_2\right)\otimes(013)\right\rbrace\]
Apply lemma \ref{4pt} (five-term relation in $\beta_2^D(F)$) then we will have
\begin{align}\label{alt1}
=\frac{2}{15}\text{Alt}_6\left\lbrace\left(-\left\llbracket\frac{r(2|1530)}{r(1|0342)}\right\rrbracket^D_2+\llbracket r(2|1053)\rrbracket^D_2-\llbracket r(1|0234)\rrbracket^D_2\right)\otimes(013)\right\rbrace
\end{align}
We will treat the above three terms individually. We consider first term now,
\[\text{Alt}_6\left\lbrace\left\llbracket\frac{r(2|1530)}{r(1|0342)}\right\rrbracket^D_2\otimes(013)\right\rbrace\]
For each individual determinant, e.g. (013), we have the following terms.
\[\text{Alt}_6\left\lbrace\left\llbracket\frac{r(2|1530)}{r(1|0342)}\right\rrbracket^D_2\otimes(013)\right\rbrace=\text{Alt}_6\left\lbrace\frac{1}{36}\text{Alt}_{(013)(245)}\left(\left\llbracket\frac{r(2|1530)}{r(1|0342)}\right\rrbracket^D_2\otimes(013)\right)\right\rbrace\]
We need a subgroup in $S_6$ which fixes (013) as a determinant i.e. $(013)\sim(310)\sim(301)\cdots$.

Here $S_3$ permuting $\{0,1,3\}$ and another one permuting $\{2,4,5\}$ i.e. $S_3\times S_3$. Now consider
\begin{align*}
&\text{Alt}_{(013)(245)}\left\lbrace\left\llbracket\frac{r(2|1530)}{r(1|0342)}\right\rrbracket^D_2\otimes(013)\right\rbrace\\
=&\text{Alt}_{(013)(245)}\left\lbrace\left\llbracket\frac{(210)(235)}{(213)(250)}\cdot\frac{(104)(132)}{(102)(135)}\right\rrbracket^D_2\otimes(013)\right\rbrace\\
=&\text{Alt}_{(013)(245)}\left\lbrace\left\llbracket\frac{(253)(104)}{(250)(134)}\right\rrbracket^D_2\otimes(013)\right\rbrace\\
&\text{ By using odd permutation (25) the above becomes}\\
=&0
\end{align*}
The new shape of (\ref{alt1}) is
\begin{align}\label{alt2}
=\frac{2}{15}\text{Alt}_6\left\lbrace\left(\llbracket r(2|1053)\rrbracket^D_2-\llbracket r(1|0234)\rrbracket^D_2\right)\otimes (013)\right\rbrace
\end{align}
Now we will consider the first terms
\begin{align}
&\frac{2}{15}\text{Alt}_6\left\lbrace\llbracket r(2|1053)\rrbracket^D_2\otimes (013)\right\rbrace\notag\\
=&\frac{2}{15}\text{Alt}_6\left\lbrace\frac{1}{6}\text{Alt}_{(245)}\llbracket r(2|1053)\rrbracket^D_2\otimes (013)\right\rbrace\notag\\
=&\frac{1}{45}\text{Alt}_6\{\Big(\llbracket r(4|1023)\rrbracket^D_2-
\llbracket r(2|1043)\rrbracket^D_2\notag\\
&\quad\quad\quad+\llbracket r(5|1043)\rrbracket^D_2-\llbracket r(4|1053)\rrbracket^D_2\notag\\
&\quad\quad\quad+\llbracket r(2|1053)\rrbracket^D_2-\llbracket r(5|1023)\rrbracket^D_2\Big)\otimes(013)\}\notag
\end{align}
We are able to use lemma \ref{4pt} (projected five-term relation in $\beta^D_2(F)$) here.
\begin{align}\label{alt3}
=\frac{1}{45}\text{Alt}_6\{\Big(&\llbracket r(0|1234)\rrbracket^D_2-\llbracket r(1|0234)\rrbracket^D_2-\llbracket r(3|0124)\rrbracket^D_2\notag\\
+&\llbracket r(0|1435)\rrbracket^D_2-\llbracket r(1|0435)\rrbracket^D_2+\llbracket r(3|0145)\rrbracket^D_2\notag\\
+&\llbracket r(0|1532)\rrbracket^D_2-\llbracket r(1|0532)\rrbracket^D_2+\llbracket r(3|0152)\rrbracket^D_2\Big)\otimes (013)\}\notag\\
&\text{Use the cycle (013)(245) then we get}\notag\\
=\frac{1}{45}\cdot9\text{Alt}_6&\left\lbrace\llbracket r(0|1234)\rrbracket^D_2\otimes(013)\right\rbrace
\end{align}
We also have $-\frac{2}{15}\text{Alt}_6\left\lbrace\llbracket r(1|0234)\rrbracket^D_2\otimes (013)\right\rbrace$ from (\ref{alt2}) which can be written as
\[\frac{1}{45}\cdot-6\text{Alt}_6\left\lbrace\llbracket r(1|0234)\rrbracket^D_2\otimes (013)\right\rbrace\]
then (\ref{alt2}) can be written as
\begin{align}
=\frac{1}{45}\text{Alt}_6\lbrace\Big(&9\llbracket r(0|1234)\rrbracket^D_2-6\left\llbracket r(1|0234)\right\rrbracket^D_2\Big)\otimes (013)\rbrace\notag
\end{align}
Use the cycle (01). We will get $\frac{1}{3}\text{Alt}_6\left\lbrace\llbracket r(0|1234)\rrbracket^D_2\otimes(013)\right\rbrace$ as a  result of (\ref{alt2}).

This gives the first term in (\ref{alt0}). For the second one, consider the second part of (\ref{alt0}) which has a $D\log$ factor in $F$ and we know that $\frac{D(ab)}{ab}=\frac{D(a)}{a}+\frac{D(b)}{b}$ and $\frac{D(\frac{a}{b})}{\frac{a}{b}}=\frac{D(a)}{a}-\frac{D(b)}{b}$, while the right factor of second term is in $\mathcal{B}_2(F)$ which is equipped with five-term relation so same procedure can be adopted for the second term as we did for first term. So, after passing through above procedure for second term, we get from the second term of (\ref{alt0})  $\frac{1}{3}\text{Alt}_6\left\lbrace D\log(013)\otimes\left[r(0|1234)\right]_2\right\rbrace$, at the end of the computation we have from the LHS of the diagram (simpler form of the diagram)
\begin{align}\label{alt4}
=\frac{1}{3}\text{Alt}_6\left\lbrace\llbracket r(0|1234)\rrbracket^D_2\otimes(013)+D\log(013)\otimes[r(0|1234)]_2\right\rbrace
\end{align}
The above allows us to rewrite $\tau^3_1$ using alternation sums. In fact, we have
\begin{align}
\tau^3_1(l_0,\ldots,l_4)=\frac{1}{3}\text{Alt}_5&\{\llbracket r(l_0|l_1,l_2,l_3,l_4)\rrbracket^D_2\otimes\Delta(l_0,l_1,l_2)\notag\\
&+D\log(\Delta(l_0,l_1,l_2))\otimes[r(l_0|l_1,l_2,l_3,l_4)]_2\}\notag
\end{align}
In reduced notation, the above can also be written as
\begin{align}
\tau^3_1(0\ldots4)=\frac{1}{3}\text{Alt}_5\left\lbrace\llbracket r(0|1234)\rrbracket^D_2\otimes(012)+D\log(012)\otimes[r(0|1234)]_2\right\rbrace\notag
\end{align}
It remains to compare $\partial\circ\tau^3_2(0,\ldots,5)$ with $\tau^3_1\circ d(0\ldots5)$. For the latter, apply cycle (012345) for $d$ and then expand Alt$_5$ from the definition of  $\tau^3_1$ so we get 
\begin{align}
\tau^3_1\circ d(0\ldots5)&=\frac{1}{3}\text{Alt}_6\left\lbrace\llbracket r(0|1234)\rrbracket^D_2\otimes(012)+D\log(012)\otimes[r(0|1234)]_2\right\rbrace\notag
\end{align}
Now use the odd permutation (23) then the above becomes
\begin{align}
=-\frac{1}{3}\text{Alt}_6\{\llbracket r(0|1324)\rrbracket^D_2\otimes(013)+D\log(013)\otimes[r(0|1324)]_2\}\notag
\end{align}
Finally use the two-term relation to get the correct sign and it will be same as (\ref{alt4}). This proves the theorem.
\end{proof}
\begin{cor}
The diagram (\ref{bicomp2}) is commutative, i.e. there is a morphism of complexes between the Grassmannian complex and a variant of  Cathelineau's complex which involves the $F$-vector spaces $\beta_3^D(F)$ and $\beta_2^D(F)$ and the groups $\mathcal{B}_2(F)$ and $F\times \bigwedge^2F^\times$.
\end{cor}
\begin{proof}
The proof follows from combining Theorem \ref{claim3a} and Theorem \ref{claim3b}.
\end{proof}

Now consider the diagram (\ref{bicomp2}) and note that $\tau_1^3\circ d\in \ker\partial^D$. It is clear from the commutativity of the diagram that $\partial^D\Big(\tau_1^3\big(d(l_0,\ldots,l_5)\big)\Big)=0$.

Goncharov has given a morphism from the Grassmannian bicomplex to $\Gamma(n)$, here we try to establish a result in the following proposition for the infinitesimal case. 
\begin{rem}\label{alld}
The following maps
\begin{enumerate}
\item $C_4(3)\xrightarrow{d'} C_3(2)\xrightarrow{\tau^2_0}F\otimes F^\times$

\item $C_5(3)\xrightarrow{d'} C_4(2)\xrightarrow{\tau^2_1}\beta^D_2(F)$

\item $C_5(4)\xrightarrow{d'} C_4(3)\xrightarrow{\tau^3_0}F\otimes \wedge^2F^\times$

\item $C_{n+1}(n+1)\xrightarrow{d'} C_{n+1}(n)\xrightarrow{\tau^n_0}F\otimes \bigwedge^{n-1}F^\times$
\end{enumerate}
are zero, where 
\begin{align*}
\tau^n_{0}(l_0,\ldots,l_{n})\notag\\
=\sum_{i=0}^{n}(-1)^{i}\Bigg(&\frac{D\left(\Delta(l_0,\ldots,\hat{l}_i,\ldots,l_{n})\right)}{\Delta(l_0,\ldots,\hat{l}_i,\ldots,l_{n})}\otimes\frac{\Delta(l_0,\ldots,\hat{l}_{i+1},\ldots,l_{n})}{\Delta(l_0,\ldots,\hat{l}_{i+2},\ldots,l_{n})}\notag\\
&\wedge\cdots\wedge\frac{\Delta(l_0,\ldots,\hat{l}_{i+(n-1)},\ldots,l_{n})}{\Delta(l_0,\ldots,\hat{l}_{i+n},\ldots,l_{n})}\Bigg),\quad i\mod (n+1)
\end{align*}
\end{rem}
\section*{Acknowledgements}
This article consists on a chapter of author's doctoral thesis at University of Durham and under the supervision of Dr. Herbert Gangl. Author would also like to thank to Prof. Spencer Bloch for his valuable comments and suggestions when he visited to Durham.

\end{document}